\documentclass[a4paper,10pt]{amsart}

\usepackage{mathrsfs}
\usepackage{amssymb,amsmath, amsfonts, amsthm}
\usepackage{latexsym}
\usepackage{color} 
\usepackage[dvips]{epsfig}
\usepackage{graphicx}

\newtheorem{defi}{Definition}[section]
\newtheorem{theorem}[defi]{Theorem}
\newtheorem{lemma}[defi]{Lemma}

\newtheorem{proposition}[defi]{Proposition}
\newtheorem{remark}[defi]{Remark}

\def\N{\mathbb{N}}
\def\R{\mathbb{R}}

\title{Local and global solutions for some parabolic nonlocal problems}
\author{Isabella Ianni}
\address{Department of Mathematics, Seconda Universita di Napoli\\ viale Lincoln 5, Caserta, Italy}
\email{isabella.ianni@unina2.it}
\begin{document}

\maketitle

\begin{abstract}
We study local and global existence of solutions for some semilinear parabolic initial boundary value problems with  autonomous nonlinearities having a ``Newtonian'' nonlocal term.
\end{abstract}

\section{Introduction}

We consider the following  semilinear  parabolic initial-boundary-value  problems (IBVP)

\begin{equation}\label{IBVP}
 \left\{
\begin{array}{l}
 \frac{\partial u}{\partial t}-\Delta u +u=F_i(u)\qquad \mbox{ in } \Omega\times (0,+\infty),\\
u=0 \qquad\qquad\qquad\qquad\;\;\;\mbox{ in } \partial\Omega\times (0,+\infty),\\
u(\cdot,0)=u_0 \qquad\qquad\qquad\;\mbox{ on }\Omega,
\end{array}
\right.
\end{equation}
$$$$
where $\Omega$ is a smoothly bounded domain in $\R^3$ and $F_i,$ $i=1,2,3$ is one of the following autonomous nonlinearities 

\begin{equation}\label{1nnlinearity}F_1(u)(x):=\phi_u(x) u(x) \qquad x\in\Omega,\end{equation}
\begin{equation}\label{2nnlinearity}F_2(u)(x):=|u(x)|^{q-1}u(x)+\phi_u(x) u(x) \qquad x\in\Omega, \end{equation}
\begin{equation}\label{nnlinearity}F_3(u)(x):=|u(x)|^{q-1}u(x)-\phi_u(x) u(x) \qquad x\in\Omega, \end{equation}
$$$$
$q\in (1,5)$ and $\phi_u$ is the ``Newtonian'' nonlocal term:

\begin{equation}\label{nnlinearity2}\phi_u(x):=\int_{\Omega}\frac{1}{|x-y|}u^2(y)dy\qquad x\in\R^3.\end{equation}

$$$$
Elliptic problems with the nonlocal term $\phi_uu$ have been object of many investigations in the last years (Schr\"odinger-Newton problem, Schr\"odinger-Poisson-Slater problem). In particular in \cite{IanniSignChanging} the author itself studies nodal solutions for the  nonlocal elliptic problem corresponding to \eqref{IBVP} with nonlinearity \eqref{nnlinearity} via a dynamical approach. From this work the interest in the parabolic problems \eqref{IBVP} arises naturally.
$$$$

Up to our knowledge, this is the first time that parabolic problems with nonlinearities involving the nonlocal term \eqref{nnlinearity2}  are  studied. Different kinds of superlinear nonlocal nonlinearities have been exploited for instance in \cite{QuittnerGen}, moreover the semilinear parabolic problem  with  the power-type nonlinearity has been extensively studied (see for instance \cite{Amann2, Amann, CazenaveLions, Quittner}).\\

As we will see a key ingredient to handle this nonlocal term is the Hardy-Littlewood-Sobolev inequality  (see Lemma \ref{Lemma1}).

This inequality is crucial  to prove the local existence of solutions (see the proof of Lemma \ref{ipotesi_Hpstausigmal_diAmann}).
Moreover thanks to the same  inequality  we are  also able to obtain a ``polynomial bound'' for the  nonlinearities which, together with an opportune  {\itshape{a priori}} bound for the solutions, eventually leads us to obtain global existence and compactness results.
\\
We anticipate that, concerning the a priori bound for the solutions,  the cases of ``combined" nonlinearity, namely nonlinearity \eqref{2nnlinearity} or \eqref{nnlinearity}, are the most delicate to be studied, since the a priori bound we need depends also on the value of $q.$ 

Indeed, as we will see,  in order to cover different values of $q,$ we need to combine the techniques we previously used  for the ``pure Newtonian" case (namely nonlinearity \eqref{1nnlinearity}) with some more refined arguments, which among other things  involve once more the Hardy-Littlewood-Sobolev inequality.
Moreover the case of nonlinearity \eqref{2nnlinearity} is particularly critical and we obtain a priori bounds only restricting the value of the exponent $q$ to the range $[3,5).$ 

$$$$

Our first main result is the following local existence and regularity theorem

$$$$

\begin{theorem}\label{teoLocal}
Let $p>3.$ 
For every $u_0\in W^{1,p}_{0}(\Omega)$ the IBVP \eqref{IBVP} has a unique $L^p$-solution $u(t)=\varphi(t,u_0)$
with maximal existence time $T:=T(u_0)>0.$

Moreover
\begin{itemize}
\item[i)] $u\in C^1((0,T),L^p(\Omega))\cap C((0,T),W^{2,p}(\Omega))\cap C^{\frac{1-\lambda}{2}}([0,T),W^{\lambda,p}(\Omega))$ for every $\lambda\in [0,1];$\\

\item[ii)] for each $t_1\in(0,T(u_0))$ the solution $u$ satisfies the integral equation
\begin{equation}\label{eqIntegrale}u(t)=e^{-(t-t_1)A_p}u(t_1)+\int_{t_1}^te^{-(t-s)A_p}F(u(s))ds\qquad t\in [t_1,T(u_0))\end{equation}
where $A_p:=-\Delta+Id:W^{2,p}_0(\Omega)\subset L^p(\Omega)\rightarrow L^p(\Omega);$\\ 

\item[iii)] the set $\mathcal G:=\{(t,u_0)\in [0,\infty)\times W_0^{1, p}(\Omega) :  t\in [0,T(u_0))\}$ is open in $[0,\infty)\times W^{1,p}_{0}(\Omega),$  $\varphi:\mathcal G\rightarrow W^{1,p}_{0}(\Omega)$ is a semiflow on $W^{1,p}_{0}(\Omega);$\\

\item[iv)] $u\in C^{\frac{2-\lambda}{2}}((0,T),W^{\lambda,p}(\Omega))$ for any $\lambda\in [1,2).$ Moreover 
for every $u_0\in W_0^{1, p}(\Omega)$ and every $t\in (0,T(u_0))$ there is a neighborhood $U\subset W_0^{1, p}(\Omega)$ of $u_0$ in $W_0^{1, p}(\Omega)$ and a positive constant $C$ such that $T(\tilde u_0)>t$ for $\tilde u_0\in U,$ and 
$$\|\varphi (t,\tilde u_0)- \varphi (t,u_0)\|_{\lambda,p}\leq C\|\tilde u_0-u_0\|_{1,p}$$

\item[v)] u is a classical solution for $t\in (0,T).$
\end{itemize}
\end{theorem}

$$$$

We denote by $E_i:W^{1,2}_{0}(\Omega)\rightarrow \mathbb R,$ $i=1,2,3$ the energy functionals in case of nonlinearity \eqref{1nnlinearity}, \eqref{2nnlinearity} and \eqref{nnlinearity} respectively

$$E_1(u):=\frac{1}{2}\int_{\Omega} (|\nabla u(x)|^2+u(x)^2)dx-\frac{1}{4}\int_{\Omega} \phi_u(x)u^2(x) dx$$
$$E_2(u):=\frac{1}{2}\int_{\Omega} (|\nabla u(x)|^2+u(x)^2)dx-\frac{1}{4}\int_{\Omega} \phi_u(x)u^2(x) dx-\frac{1}{q+1}\int_{\Omega}  |u(x)|^{q+1}dx$$
$$E_3(u):=\frac{1}{2}\int_{\Omega} (|\nabla u(x)|^2+u(x)^2)dx+\frac{1}{4}\int_{\Omega} \phi_u(x)u^2(x) dx-\frac{1}{q+1}\int_{\Omega}  |u(x)|^{q+1}dx.$$

$$$$
Next results are about global existence and compactness in case of nonlinearity \eqref{1nnlinearity}, \eqref{2nnlinearity} and \eqref{nnlinearity} respectively:

$$$$

\begin{theorem}\label{1TeoGlobalSolution} Let $p>3.$\\
Let $u_0\in W^{1,p}_0(\Omega),$ $u(t)=\varphi (t, u_0)$ be the solution of \eqref{IBVP} and \eqref{1nnlinearity} and $T=T(u_0).$ \\
If $$t \mapsto E_1(u(t))\ \mbox{ is bounded from below on }(0,T)$$ then 
\begin{itemize} 
\item $T=+\infty$
\item
for every $\delta>0,$ 
$\sup_{t\geq\delta}\|u(t)\|_{s,p}<\infty$ for every $s\in [1,2)$
and 
the set $\{u(t):t\geq\delta\}$ is  relatively compact in $C^1(\bar \Omega).$
\end{itemize}
\end{theorem}

$$$$

\begin{theorem}\label{2TeoGlobalSolution} Let $p>3.$\\
Let $u_0\in W^{1,p}_0(\Omega),$ $u(t)=\varphi (t, u_0)$ be the solution of \eqref{IBVP} and \eqref{2nnlinearity} and $T=T(u_0).$ \\
Let $q\in (1, 2^*-1)$. If $$t \mapsto E_2(u(t))\ \mbox{ is bounded from below on }(0,T)$$ then 
the conclusions of Theorem \ref{1TeoGlobalSolution} are true.\end{theorem}

$$$$

\begin{theorem}\label{TeoGlobalSolution} Let $p>3.$\\
Let $u_0\in W^{1,p}_0(\Omega),$ $u(t)=\varphi (t, u_0)$ be the solution of \eqref{IBVP} and \eqref{nnlinearity} and $T=T(u_0).$ \\
Let $q\in [3, 2^*-1)$. If $$t \mapsto E_3(u(t))\ \mbox{ is bounded from below on }(0,T)$$ then 
the conclusions of Theorem \ref{1TeoGlobalSolution} are true.\end{theorem}

$$$$

The paper is organized as follows.
\\
Section 2 is for notations and preliminaries, in particular we recall  the Hardy-Littlewood-Sobolev inequality involving the nonlocal term (see Lemma \ref{Lemma1}).
\\
Section 3 is related to the proof of the local existence result.
\\
The main section of this paper is Section 4. Here we prove polynomial bounds for the nonlinearity $F_i$ (Lemma  \ref{stimanonlinearitap}) and a priori bounds for the solutions (Proposition \ref{aPriori}).
\\
Last Section 5 contains the proof of Theorem \ref{1TeoGlobalSolution},  Theorem \ref{2TeoGlobalSolution} and  Theorem \ref{TeoGlobalSolution}.
$$$$
\section{Notations and preliminaries}
Let us fix some notations.
\\
$L^r(\Omega),$ $W^{s,r}(\Omega),$ $W^{s,r}_0(\Omega)$ are the usual Lebesgue and the Sobolev or Sobolev-Slobodeckii spaces. 
\\
$C^{k+\alpha}(\bar \Omega),$ with $k\in \N$ and $\alpha\in (0,1)$ is the Banach space of all the functions belonging to $C^{k}(\bar\Omega)$ whose $k$-th order partial derivatives are  uniformly $\alpha$-H\"older continuous in $\bar\Omega.$ 
\\
We denote by $\|\cdot\|_r, \|\cdot\|_{s,r},\|\cdot\|_{C^{k+\alpha}}$   the usual norms in $L^r(\Omega),$ $W^{s,r}(\Omega), C^{k+\alpha}(\bar \Omega)$ respectively.\\
Let $p'$ denote the dual exponent to $p\in (1,\infty),$ namely $\frac{1}{p}+\frac{1}{p'}=1.$\\
By $C$ we denote various positive constants which may vary from step to step.

$$$$

We shall use the following interpolation embeddings
\begin{equation}\label{interpolationTheorem1} W^{1,2}(J,L^2(\Omega))\cap L^{\beta(q+1)}(J, L^{q+1}(\Omega))\hookrightarrow L^{\infty}(J, L^{a}(\Omega)),
\end{equation}
\begin{equation}\label{interpolationTheorem2} W^{1,2}(J,L^2(\Omega))\cap L^{4}(J, W^{1,2}(\Omega))\hookrightarrow L^{\infty}(J, L^{\rho}(\Omega)),
\end{equation}

where $J$ is  a compact interval, $a\in (1, q+1-\frac{q-1}{\beta+1})$ and $\rho\in (1, \frac{18}{5})$
(for the proof see \cite[Appendix]{CazenaveLions} or \cite{QuittnerGen}).

$$$$
We also recall that the maximal Sobolev regularity property holds for problem \eqref{IBVP} (see \cite[p. 188]{AmannBook} or \cite[Appendix E, p. 470]{QuittnerSoupletBook}). Hence, given a compact interval $J$ and $F_i\in L^a(J, L^b(\Omega))$, $1<a,b<\infty,$ the solution $u$ of \eqref{IBVP} satisfies 

\begin{equation}\label{MaximalSobolev} \|u\|_{W^{1,a}(J, L^b(\Omega))}+\|u\|_{L^a(J, W^{2,b}(\Omega))}\leq C\left(\|u_0\|_{W^{s,b}(\Omega)}+\|F_i\|_{L^a(J, L^b(\Omega))}\right),
\end{equation}

provided $s>2(1-1/a).$

$$$$

The following lemma will play a crucial role in the proof of the a priori estimates in Section 4

\begin{lemma}\label{Lemma1}
If $u\in L^{2m}(\Omega)$ then $\phi_u$ is well-defined and belongs to $L^r(\R^3)$ with
\begin{equation}\label{HLS_nel_mio_caso}\|\phi_u\|_{L^r(\R^3)}\leq C\|u\|^2_{2m},\quad 
\frac{1}{m}+\frac{1}{3}=1+\frac{1}{r},\quad 1<m,r<\infty.\end{equation}
\end{lemma}
\begin{proof} 
Let $\tilde u$ the trivial extension (to $0$) of $u$ in $\R^3,$ then $\tilde u\in L^{2m}(\R^3)$ and
by Hardy-Littlewood-Sobolev inequality (see \cite[Theorem 4.3]{LiebLoss}) one has the following
$$\left\|\left(\frac{1}{|x|}\ast \tilde u^2\right)\right\|_{L^r(\R^3)}\leq C\|\tilde u\|^2_{L^{2m}(\R^3)},\quad \frac{1}{m}+\frac{1}{3}=1+\frac{1}{r},\quad 1<m,r<\infty.$$
The conclusion comes observing that $\|\tilde u\|_{L^{2m}(\R^3)}=\|u\|_{2m}$ and that
$$\left(\frac{1}{|x|}\ast \tilde u^2\right)(x)=\int_{\R^3}\frac{1}{|x-y|}\tilde u^2(y)dy=\int_{\Omega}\frac{1}{|x-y|}u^2(y)dy=\phi_u(x).$$
\end{proof}

%
$$$$
\section{Proof of Theorem \ref{teoLocal}} 

\begin{lemma}\label{ipotesi_Hpstausigmal_diAmann}
 Let $p>3$ 
%
%
and $\tau>\frac{3}{p}$ then $F_i:W^{\tau,p}(\Omega)\rightarrow L^{p}(\Omega),$ $i=1,2,3$ is locally Lipschitz continuous.
\end{lemma}
\begin{proof}
That $ W^{\tau,p}(\Omega)\ni u \mapsto |u|^{q-1}u \in L^p(\Omega)$ is well defined and locally Lipschitz can be found in \cite[Proposition 15.4]{Amann2}. 
\\
Here we show that $W^{\tau,p}(\Omega)\ni u \mapsto \phi_uu$ is well defined and locally Lipschitz continuous.

Let $u\in W^{\tau,p}(\Omega),$ first we prove that $\phi_uu\in L^p(\Omega).$ 
Indeed from the continuous embedding $W^{\tau,p}(\Omega)\hookrightarrow C^0(\bar{ \Omega})$ (since $\tau>\frac{3}{p}$) and using \eqref{HLS_nel_mio_caso} with $m:=\frac{3p}{2p+3}$ (indeed $m>1$ since $p>3$) we get
\begin{eqnarray*}
\|\phi_uu\|_p & \leq &\|u\|_{C^0}\|\phi_u\|_p\\
 &\leq & C\|u\|_{\tau,p}\|\phi_u\|_p \\
& \leq & C\|u\|_{\tau,p}\|u\|^2_{2m}\\
& \leq & C\|u\|_{\tau,p}\|u\|^2_{C^0}\\
&\leq & C\|u\|^3_{\tau,p}<\infty.
\end{eqnarray*}
Next, using again inequality \eqref{HLS_nel_mio_caso} with $m:=\frac{3p}{2p+3}$ and the continuous embedding $W^{\tau,p}(\Omega)\hookrightarrow C^0(\bar{ \Omega}),$ we prove the locally Lipschitz continuity. Indeed letting $u_i\in W^{\tau,p}(\Omega),$ $i=1,2$ $\|u_i\|_{\tau,p}\leq R,$ then
\begin{eqnarray*}
\|u_1\phi_{u_1}-u_2\phi_{u_2}\|_p &\leq& \|\phi_{u_1}(u_1-u_2)\|_p+\|u_2(\phi_{u_1}-\phi_{u_2})\|_p\\
&\leq& \|u_1-u_2\|_{C^0}\|\phi_{u_1}\|_p+ \|u_2\|_{C^0}\|\phi_{u_1}-\phi_{u_2}\|_p\\
&\leq& \|u_1-u_2\|_{\tau,p}\|\phi_{u_1}\|_p+ \|u_2\|_{\tau,p}\|\phi_{w}\|_p\qquad(\mbox{where }w^2:=|u_1^2-u_2^2|)\\
&\leq& C\|u_1-u_2\|_{\tau,p}\|u_1\|^2_{2m}+ C\|u_2\|_{\tau,p}\|u_1^2-u_2^2\|_m\qquad \mbox{(by \eqref{HLS_nel_mio_caso})}\\
&\leq& C\|u_1-u_2\|_{\tau,p}\|u_1\|^2_{C^0} + C\|u_2\|_{\tau,p}\|(u_1-u_2)(u_1+u_2)\|_m\\
&\leq& C\|u_1-u_2\|_{\tau,p}\|u_1\|^2_{\tau,p} + C \|u_2\|_{\tau,p}(\|u_1\|_{C^0}+\|u_2\|_{C^0})\|u_1-u_2\|_{C^0} \\
&\leq& C\|u_1-u_2\|_{\tau,p}. 
\end{eqnarray*} 
\end{proof}

$$$$
Next we prove Theorem \ref{teoLocal}.
$$$$

\begin{proof}[Proof of Theorem \ref{teoLocal}]
The first part of the proof can be derived from abstract results of Amann concerning local existence and regularity for semilinear parabolic IBVP (see \cite{Amann2}).
\\
Following Amann's notation we set 
$\mathscr A(t)\equiv \mathscr A=-\Delta +Id$ and
$\mathscr B(t)\equiv\mathscr B=Id.$ 
\\
Since $p>3,$  we can choose  $\tau\in (\frac{3}{p},1)$  in Lemma \ref{ipotesi_Hpstausigmal_diAmann}. As a consequence the hypothesis 
``$\mathscr H(p,s,\tau,\sigma,l)$'' in \cite{Amann2}  with values $s=0,$ $\sigma=1$
and $l\geq2$ is satisfied.
Hence the local existence, the integral equation \eqref{eqIntegrale} and the semiflow properties of $\varphi$  follow directly applying for each fixed $T>0$ Theorem 15.1 and Corollary 15.2 in \cite{Amann2} (we point out that $\mathscr A,$ $\mathscr B$ and  $F_i$ are not depending on $t$).
\\ 
The further regularity  as well as the continuous dependence property in stronger norm in point iv) can be derived from Theorem 51.7-Example 51.4 in the Appendix E of the book \cite{QuittnerSoupletBook} (see in particular Remark 51.8 (iii)).
\\
Last to prove that the solution is classical one can adapt the arguments in the Example 51.9 of \cite{QuittnerSoupletBook}, 
 we give here only a brief sketch of it and we refer to \cite{QuittnerSoupletBook} for further details.
Basically one considers the parabolic problem as a linear problem  with H\"older continuous right-hand side and applies parabolic Schauder estimates.
It's not difficult to show the H\"older continuity of the RHS once one knows that  $u\in C^{\rho}((0,T),C^{1+\rho}(\bar \Omega))$ for a certain $\rho\in (0,1).$ And this last result  follows from point iv) choosing $\lambda:=2-2\rho,$ for $\rho \in (0,1)$ and sufficiently small in order to have the embedding $W^{\lambda,p}(\Omega)\hookrightarrow C^{1+\rho}(\bar \Omega).$ 
%
%
%
\end{proof}

$$$$
\begin{remark}
We point out that the local existence results in Theorem \ref{teoLocal} hold actually for every $q\geq 1$ and not only for $q\in (1,5).$
\end{remark}

$$$$

\section{A priori estimates for solutions}
Throughout this section  $p>3.$ Moreover we consider a fixed  $u_0\in W^{1,p}_0(\Omega)$ and, for $i=1,2,3$ we let   $u_i(t)=\varphi_i (t, u_0)$  be the solution of \eqref{IBVP} with the nonlinearity $F_i$ defined respectively in \eqref{1nnlinearity}, \eqref{2nnlinearity} and \eqref{nnlinearity} and $T=T(u_0)$ be its maximal existence time. If no confusion seems likely, then we may shortly write $u$ instead of $u_i$ and also $F, E$ instead of $F_i, E_i.$

$$$$

Next Lemma gives a polynomial bound for the nonlinearity $F_i$ and it is proved using Lemma \ref{Lemma1}. 
$$$$
\begin{lemma}\label{stimanonlinearitap}
Let $1<r<\infty.$ 
Then

\begin{equation}\label{laPrima}\|F_1(v)\|_r\leq C\||v|^{3}\|_r\quad \forall\ v\in L^{3r}(\Omega);\end{equation}
moreover for $i=2,3$
\begin{equation}\label{laSeconda}\|F_i(v)\|_r\leq C\left(1+\||v|^{\kappa}\|_r\right)\quad\forall v\in L^{\kappa r}(\Omega)\ \end{equation}
where  $\kappa:=\max\{q,3\}.$
\end{lemma}
\begin{proof}
First we prove \eqref{laPrima}. We fix $\alpha> \max\{\frac{3}{2},\frac{3}{r}\}$ and we define $m=\frac{3r\alpha}{2r\alpha+3}$ ($m>1$ since $\alpha>\frac{3}{r}$). From H\"older inequality, Lemma \ref{Lemma1} and Sobolev embeddings we obtain
\begin{eqnarray*}\nonumber \|F_1(v)\|_r &= & \|v(|x|^{-1}\ast v^2)\|_r\\\nonumber
&\stackrel{\scriptsize{\mbox{H\"older}}}{\leq}& C \|v\|_{r\alpha'}\  \|(|x|^{-1}\ast v^2)\|_{r\alpha}\\
&\stackrel{\scriptsize{\mbox{Lemma \ref{Lemma1}}}}{\leq} &C \|v\|_{r\alpha'}\  \|v\|^2_{2m}\\
&\stackrel{\scriptsize{\mbox{Sobolev emb.}}}{\leq}  & C\|v\|_{3r}\  \|v\|^2_{3}\\
&\stackrel{\scriptsize{\mbox{Sobolev emb.}}}{\leq}  &  C\|v\|^{3}_{3r}.
\end{eqnarray*}
where we have used the fact that $2m<3$ and $\alpha'<3$ (since by definition $\alpha>\frac{3}{2}$).
\\
Inequality \eqref{laSeconda} follows in a similar way, indeed for $i=2,3$ one has
\begin{eqnarray}\nonumber \|F_i(v)\|_r &\leq & \||v|^q\|_r + \|v(|x|^{-1}\ast v^2)\|_r\\\nonumber
&\stackrel{\scriptsize{\mbox{H\"older}}}{\leq}&\||v|^q\|_r + C \|v\|_{r\alpha'}\  \|(|x|^{-1}\ast v^2)\|_{r\alpha}\\
\label{qui}
&\stackrel{\scriptsize{\mbox{Lemma \ref{Lemma1}}}}{\leq} &\|v\|^q_{qr} + C \|v\|_{r\alpha'}\  \|v\|^2_{2m} \\\nonumber
&\stackrel{\scriptsize{\mbox{Sobolev emb.}}}{\leq} &\|v\|^q_{\kappa r} + C \|v\|_{\kappa r}\  \|v\|^2_{\kappa} \\\nonumber
&\stackrel{\scriptsize{\mbox{Sobolev emb.}}}{\leq}  & \|v\|^q_{\kappa r} + C \|v\|^3_{\kappa r}\  \\\nonumber
&\leq & \tilde C (1+\|v\|_{\kappa r}^{\kappa})\nonumber
\end{eqnarray}

where we have used the fact that $\alpha'< \kappa $ (indeed $\alpha'<3$)  and also that  $2m<\kappa$ (since $2m<3$).
\end{proof}
%
%
%

$$$$
$$$$

Next result is an a priori bound for $u_i,$ $i=1,2,3$ in a proper $L^a$-norm under the additional condition that the energy functional $E_i$ stays bounded from below along the trajectory. 
$$$$

\begin{proposition}\label{aPriori} Let $i=1,2$ or 3 and assume that $t \mapsto E_i(u_i(t))$  is bounded from below on $(0,T).$ Moreover let $q\in(1,2^*-1)$ when $i=2$ and $q\in [3,2^*-1)$ when $i=3.$
Then
\begin{equation}\label{1StimaNormaMigliore}
\sup_{(0,T)}\|u_i(t)\|_{a}<\infty,\qquad \mbox{ for all }\ a<\frac{18}{5}.
\end{equation}
 Moreover for $i=2,3$ and  $q\geq 3$ if $i=2,$ $q>3$ if $i=3,$ one has for every $\delta >0$
\begin{equation}\label{2StimaNormaMigliore}
\sup_{[\delta,T)}\|u_i(t)\|_{a}<\infty,\qquad \mbox{ for all }\ a< q+1.    
\end{equation}
\end{proposition}
$$$$
\begin{remark}\label{remarkDifferenze}
\begin{itemize}
\item[(i)] A priori estimates for solutions of parabolic equations with a power-type nonlinearity have been proved among others by \cite{CazenaveLions, Quittner}. Moreover we refer to \cite{QuittnerGen} for a priori bounds relating to solutions of more general superlinear parabolic problems, subcritical where also some nonlocal problem has been studied. Anyway as far as we know Proposition \ref{aPriori} is the first  result in this direction for nonlinearities that involve a Newtonian nonlocal  term  \eqref{nnlinearity2}.

We point out that differently with \cite{CazenaveLions, Quittner} here we are not assuming that the solution is a priori global (indeed in next section we will need  estimates \eqref{1StimaNormaMigliore} and \eqref{2StimaNormaMigliore}  exactly to prove that the solutions are global).

As a consequence we need here to impose the additional condition on the lower bound of the energy along the trajectory (a condition that  is instead obtained for free in \cite{CazenaveLions, Quittner} by a blow-up argument for global solutions).

For this reason the bound for $\|u_i(t)\|_a$ we obtain in Proposition \ref{aPriori} can not depend simply  on the norm of the initial condition  (and on $\delta$ of course) like in \cite{CazenaveLions} 
%
%
or  \cite{Quittner}
%
%
, but its dependence  on the initial condition $u_0$ must be more complicated.
Hence in our proof we fix $u_0$ and we do not investigate the way the bound depends on its norm. 

We underline that in \cite{QuittnerGen}  no condition on the lower bound of the energy is imposed and the solutions are not assumed to be a priori global. Anyway a blow-up argument gives in this case a lower bound on the energy only  for $t<T-\lambda$ where $\lambda>0$ and so  the a priori bounds obtained in \cite{QuittnerGen} are of the form $\|u(t)\|_a<C_{\lambda}$ for $t<T-\lambda.$

Last we remark that the assumption we  add on the lower bound of the energy is totally reasonable, see for instance \cite{IanniSignChanging} where the trajectories under consideration lie in sets on which the energy stays strictly positive.\\

\item[(ii)]
Observe that when $i=3$ no a priori bound is obtained for $q<3,$ for this reason in  Theorem \ref{TeoGlobalSolution} we will restrict to the case $q\geq 3.$
 
Moreover for both the  ``combined" nonlinearity, namely when $i=2,3,$ the bound \eqref{2StimaNormaMigliore} holds only when $q$ is big enough ($q\geq 3$ and $q>3$ respectively), while for the other values of $q$ we only have the weaker bound \eqref{1StimaNormaMigliore} (same bound as in the ``pure Newtonian" case). Anyway, 
as we will see in Section 5, in order to obtain global existence the bound \eqref{1StimaNormaMigliore} is sufficient when considering small $q$ (precisely $q<17/5$), while only for bigger values of $q$ we need a stronger bound like \eqref{2StimaNormaMigliore}. Hence, combining both the results we are eventually able to cover all the different values of $q$ considered in Theorem \ref{2TeoGlobalSolution} and \ref{TeoGlobalSolution} respectively.\\

\item[(iii)] 
The proof is divided into two main parts. The first part, concerning the case of a priori global solutions ($T=\infty$) follows quite immediately from a monotonicity property of the energy functional along the flow. The second part, concerning the case $T<\infty,$ needs a more careful analysis. 

Precisely the proof of \eqref{1StimaNormaMigliore} is mainly inspired to arguments in \cite{CazenaveLions, Quittner}.
%
 
To prove the stronger bound \eqref{2StimaNormaMigliore} we need a bootstrap argument involving the maximal Sobolev regularity property in the spirit of \cite{Quittner, QuittnerGen}. Besides due to the presence of the nonlocal Newtonian term we will need once more the Hardy-Littlewood-Sobolev inequality (see the proof of Lemma  \ref{proofTinfinito}).
%
%
\end{itemize}
\end{remark}
$$$$

In order to prove Proposition \ref{aPriori} we need the following lemma\\

\begin{lemma}
Let $i=1,2,3.$ The function $t\mapsto E_i(u_i(t))$ is continuous in $[0,T)$ and $C^1$ in $(0,T),$ moreover
\begin{equation}\label{uno}
 \frac{d}{dt}E_i(u_i(t))=-\|{u_i}_t(t)\|_2^2\quad\mbox{ for }t\in (0,T),
\end{equation}

\begin{equation}\label{quattroMia}
E_i(u_i(t))\leq E_i(u_0)\quad  \mbox{ for } t\in[0,T).
\end{equation}
Furthermore
\begin{equation}\label{sei}
 \|u_i(t)\|_{1,2}^2\leq C(1+\|u_i(t){u_i}_t(t)\|_1) \quad\mbox{ for }t\in(0,T),\end{equation}
where when $i=3$ we are also assuming that $q\geq 3.$ 
\\
Moreover, for $i=2,3$ 
 \begin{equation}\label{seiBIS}
 \|u_i(t)\|_{q+1}^{q+1}\leq C(1+\|u_i(t){u_i}_t(t)\|_1) \quad\mbox{ for }t\in(0,T),
\end{equation}
where when $i=3$ we are also assuming that $q>3.$
\end{lemma}
\begin{proof}
The continuity and the differentiability of $t\mapsto E_i(u_i(t))$ follow  from the embedding $W^{1,p}_0(\Omega)\hookrightarrow W^{1,2}_0(\Omega)$ and the fact $ u_i\in C^0([0,T), W^{1,p}_0(\Omega))$ and also that $u_i$ is a classical solution on $(0,T)$.
The \eqref{uno} follows differentiating and integrating by parts, for instance in case $i=3$ we get
\begin{eqnarray*}
\dot E_3
&=&\frac{d}{dt}E_3(u_3)=
\frac{d}{dt}\int_{\Omega}\left(\frac{1}{2}|\nabla u_3|^2+\frac{1}{2}|u_3|^2-\frac{1}{q+1}|u_3|^{q+1}+\frac{1}{4}\phi_{u_3}u_3^2 \right)dx\nonumber\\
&=& \int_{\Omega}\left(\nabla u_3\nabla {u_3}_t +uu_t-|u_3|^{q-1}u_3{u_3}_t+\phi_{u_3}u_3{u_3}_t\right)dx\nonumber\\
&=& \int_{\Omega}\left(-\Delta u_3+u_3-|u_3|^{q-1}u_3+\phi_{u_3}u_3\right){u_3}_tdx\nonumber\\
&=& -\int_{\Omega}{u_3}_t^2 dx=-\|{u_3}_t\|_2^2.
\end{eqnarray*}
When $i=1,2$ the proof is similar.\\
\\
To prove \eqref{quattroMia} just observe that $t\mapsto E_i(u_i(t))$ is decreasing for $t>0$ (by \eqref{uno}) and continuous in $t=0.$
\\\\
Next we  prove \eqref{sei} in the three cases $i=1,2,3$. 
\\\\
Multiplying the equation ${u_1}_t=\Delta u_1- u_1+\phi_{u_1}u_1$ by $u_1$ and integrating by parts we obtain
$$\|u_1(t)\|_{1,2}^2\leq  \|u_1(t){u_1}_t(t)\|_1 + 4E(u_1(t)) ,$$
hence the conclusion in case $i=1$ follows from \eqref{quattroMia}.
\\\\
Similarly, multiplying the equation ${u_2}_t=\Delta u_2- u_2+|u_2|^{q-1}u_2+\phi_{u_2}u_2$ by $u_2$ and integrating by parts we obtain for any constant $K>2$
\begin{eqnarray*}\|u_2(t)\|_{1,2}^2 &\leq&  \frac{2}{K-2}\|u_2(t){u_2}_t(t)\|_1 +\frac{2K}{K-2} E(u_2(t))- \frac{2(q+1-K)}{(K-2)(q+1)}\|u_2(t)\|_{q+1}^{q+1}\\
&&-\frac{4-K}{2(K-2)}\int_{\Omega}\phi_{u_2}(t)u_2^2(t),\end{eqnarray*}
hence taking $K:= \min\{4,q+1\},$ since $\phi_{u_2}\geq 0$ by definition, we have
$$\|u_2(t)\|_{1,2}^2\leq  \frac{2}{K-2}\|u_2(t){u_2}_t(t)\|_1 +\frac{2K}{K-2} E(u_2(t))$$
and the conclusion in case $i=2$ follows again from \eqref{quattroMia}.
\\\\ 
Last we consider the case $i=3.$ In this case, multiplying the equation ${u_3}_t=\Delta u_3- u_3+|u_3|^{q-1}u_3-\phi_{u_3}u_3$ by $u_3$ and integrating by parts we obtain
for any constant $K>2$
\begin{eqnarray}\label{generaleK}\|u_3(t)\|_{1,2}^2 &\leq&  \frac{2}{K-2}\|u_3(t){u_3}_t(t)\|_1 +\frac{2K}{K-2} E(u_3(t))- \frac{2(q+1-K)}{(K-2)(q+1)}\|u_3(t)\|_{q+1}^{q+1}\nonumber\\ &&+\frac{4-K}{2(K-2)}\int_{\Omega}\phi_{u_3}(t)u_3^2(t),\end{eqnarray}
hence, taking any $K\in [4,q+1]$ ($q\geq 3$ by assumption) we get
$$\|u_3(t)\|_{1,2}^2\leq  \frac{2}{K-2}\|u_3(t){u_3}_t(t)\|_1 +\frac{2K}{K-2} E(u_3(t)),$$
and so the conclusion comes from \eqref{quattroMia}.
\\
\\
The proof of \eqref{seiBIS} follows in a similar way. 

For $i=2$ we multiply the equation ${u_2}_t=\Delta u_2- u_2+|u_2|^{q-1}u_2+\phi_{u_2}u_2$ by $u_2,$  integrate by parts and use \eqref{sei} obtaining directly
\begin{eqnarray*}\|{u_2}(t)\|_{q+1}^{q+1}&&\leq \|u_2(t){u_2}_t(t)\|_{1}+\|u_2(t)\|_{1,2}^2-\int_{\Omega}\phi_{u_2}(t)u_2^2(t)\\
&&\leq \|u_2(t){u_2}_t(t)\|_{1}+\|u_2(t)\|_{1,2}^2\\
&&\leq C(1+\|u_2(t){u_2}_t(t)\|_{1}).\end{eqnarray*}

The case $i=3$  follows from \eqref{generaleK} taking any $K\in [4, q+1)$ (indeed $q>3$ by assumption)
and using once more \eqref{quattroMia}.
%
%
%
%
%
%
\end{proof}
$$$$
$$$$
\begin{proof}[Proof of Proposition \ref{aPriori}. Case $T=\infty$]
$\;$\\\\

The $C^1$ function  $(0, +\infty)\ni t\mapsto E_i(u_i(t))$ is decreasing because of \eqref{uno} and bounded from below by assumption, hence $E_i(u_i(t))\searrow c\in\mathbb R$  and $\frac{d}{dt}E_i(u_i(t))\rightarrow 0$ as $t\rightarrow +\infty.$\\
Combining this with \eqref{uno} and \eqref{sei},  we get from H\"older inequality $$\|u_i(t)\|_{1,2}^2\leq C(1+\|u_i(t)\|_2\|{u_i}_t(t)\|_2)\leq C(1+\|u_i(t)\|_{1,2})\mbox{, for }t\in(0, +\infty)$$ namely
$$\|u_i(t)\|_{1,2}\leq C\quad \forall t\in[0,+\infty),$$
(we can consider $t=0$ because $u_i\in C^0([0,T), W^{1,2}_0(\Omega))$)
which implies, by Sobolev embedding  
$$\|u_i(t)\|_{2^*}\leq C\quad \forall t\in[0,+\infty).$$
and hence \eqref{1StimaNormaMigliore} and  \eqref{2StimaNormaMigliore} for $i=1,2,3$ respectively.
\end{proof}
$$$$
$$$$
From now on we consider $T<\infty.$
$$$$
\begin{lemma}
Assume that $t \mapsto E_i(u_i(t))$ is bounded from below on $(0,T).$ Let also $T<\infty.$ Then
\begin{equation}\label{tre}
 \sup_{t\geq 0}\|u_i(t)\|_2<\infty
\end{equation}

\begin{equation}\label{cinque}
 \int_{0}^{T}\|{u_i}_t(t)\|_2^2dt\leq C
\end{equation}
\end{lemma}

\begin{proof}
We put $E_{inf}:=\inf_{t\geq 0} E_i(u_i(t))>-\infty.$
The following argument is similar to the one used in \cite[p. 89]{WeiWeth}. Let $h(t):=\|u_i(t)\|_2^2.$ By H\"older inequality and \eqref{uno} we have for $t\in (0,T)$
$$\frac{d}{dt}h=2\int_{\Omega}u_i{u_i}_tdx\leq 2\|u_i\|_2\|{u_i}_t\|_2\leq \|u_i\|_2^2+\|{u_i}_t\|_2^2=h-\dot E_i$$ and so 
$$\frac{d}{dt}(e^{-t}h(t))=e^{-t}[\dot h-h]\leq -e^{-t}\dot E_i\leq -\dot E_i.$$ Integrating in $0<\delta< t$ we have

\begin{eqnarray*} 
e^{-t}h(t)&&=e^{-\delta}h(\delta)+\int_{\delta}^t\frac{d}{ds}(e^{-s}h(s))ds\leq e^{-\delta}h(\delta)+\int_{\delta}^t-\dot E_i(u_i(s))ds\\
 	&&=
e^{-\delta}h(\delta)-E_i(u_i(t))+E_i(u_i(\delta))\leq e^{-\delta}h(\delta)-E_{inf}+E_i(u_i(\delta))
\end{eqnarray*}
Passing to the limit for $\delta\rightarrow 0^+$ we obtain
$$e^{-t}h(t)\leq h(0)-E_{inf}+E_i(u_0):=C,$$
namely $$h(t)\leq e^tC\leq e^TC:= C_{T}\mbox{ for }t\in[0,T),$$
which proves \eqref{tre}.
The bound in \eqref{cinque} follows immediately from \eqref{uno}:
$$\int_{0}^{T}\|{u_i}_t(t)\|_2^2dt= \int_{0}^{T}(-\dot E_i(t))dt=E_i(u_0)-\lim_{t\rightarrow T}E_i(t)\leq E_i(u_0) -E_{inf}=C(>0).$$
\end{proof}

$$$$
\begin{lemma}\label{proofTinfinito} Let $i=2$ or 3 and $q\geq 3$ if $i=2,$ $q>3$ if $i=3.$ Assume also that $t \mapsto E_i(u_i(t))$  is bounded from below on $(0,T)$ and that $T<\infty.$
Let $\beta\geq 2$ and $\delta>0.$ Then

%
\begin{equation}\label{ottoIMPR}
 \int_{\delta}^{T}\|u_i(t)\|_{q+1}^{\beta(q+1)}dt<\infty
\end{equation}
and 

\begin{equation}\label{dieci}
\sup_{[\delta,T)}\|u_i(t)\|_{a}<\infty, \qquad\mbox{ for all }a<a(\beta),
\end{equation}
where $$a(\beta)=q +1-\frac{q-1}{\beta+1}.$$ 
\end{lemma}
\begin{proof}
We divide the proof in the following steps:
\begin{itemize}
%
%
\item[STEP 1] We show that \eqref{ottoIMPR} implies \eqref{dieci}.
\item[STEP 2] We prove that \eqref{ottoIMPR} is true for $\beta=2$ and $\delta=0$.
\item[STEP 3] We prove the validity of \eqref{ottoIMPR} for any $\beta>2$ (and any $\delta>0$) through a bootstrap argument.
\end{itemize}
$$$$
STEP 1. Putting together \eqref{cinque} and \eqref{ottoIMPR} we have
\begin{equation}\label{intEstuno} \int_{\delta}^{T}\left(\|{u_i}_t(t)\|_2^{2}+\|u_i(t)\|_{q+1}^{\beta(q+1)}\right)dt<\infty,
\end{equation}
and so \eqref{dieci} is a consequence of the interpolation embedding \eqref{interpolationTheorem1}.
$$$$
STEP 2.
We prove that \eqref{ottoIMPR} holds with $\beta=2$ and $\delta=0.$ From \eqref{seiBIS} (indeed we are assuming $q>3$ when $i=3$) and using H\"older inequality we have $$\|u_i(t)\|_{q+1}^{q+1}\leq C(1+\|u_i(t){u_i}_t(t)\|_1)\leq C(1+\|u_i(t)\|_2\|{u_i}_t\|_2)\quad\forall t>0,$$
so, raising it to the power two and using \eqref{tre} it follows
$$\|u_i(t)\|_{q+1}^{2(q+1)}\leq C(1+\|u_i(t)\|_2^2\|{u_i}_t\|_2^2)\leq C(1+\|{u_i}_t(t)\|_2^2)\quad\forall t>0.$$
Integrating, using \eqref{cinque} and the assumption $T<\infty$ we get the conclusion
$$\int_0^{T}\|u_i(t)\|_{q+1}^{2(q+1)}dt\leq C\int_0^{T}(1+\|{u_i}_t(t)\|^2_2)dt<\infty.$$
\\
STEP 3. 
We want  to prove that the validity of \eqref{ottoIMPR} for some $\beta\geq 2$ implies \eqref{ottoIMPR}  for some $\tilde\beta>\beta$. Moreover we want   the difference $\tilde\beta-\beta$ to be bounded below by a positive uniform constant for all $\beta$ so that, performing a bootstrap procedure, after finitely many steps, we end up with some $\tilde{\beta}$ big enough.\\
\\
First observe that from inequality \eqref{laSeconda} taking $r:=\frac{q+1}{q}$ it follows
\begin{equation}\label{usala}
\|F_i(u_i(t))\|_{\frac{q+1}{q}} \leq C\left(1+\|\, |u_i(t)|^{\kappa}\,\|_{\frac{q+1}{q}}\right)=
C \left(1+\|u_i(t)\|_{q+1}^{\frac{q}{q+1}(q+1)}\right) 
,
\end{equation}
where $\kappa:=\max\{q,3\}=q,$ since we assumed $q\geq 3.$\\

Hence the proof follows applying to our case the abstract bootstrap Lemma 2.2-(i) in \cite{QuittnerGen}. Indeed, our problem possesses the  maximal Sobolev regularity property (see \eqref{MaximalSobolev}) and so, taking $\mathcal G(u):=C(1+\|u\|_{q+1}^{q+1})$ for a suitable positive constant $C,$ one can show that the inequality \eqref{usala} above, together with \eqref{seiBIS}, \eqref{ottoIMPR} and \eqref{cinque},  implies that $\mathcal G(u_i(t))$ satisfies all the assumptions required in \cite{QuittnerGen}. 
\\
Anyway for completeness we briefly repeat here the main points of the proof adapting the abstract setting of \cite{QuittnerGen} to our case.\\

Hence, fix any $\delta>0$ and let \eqref{ottoIMPR}  be true for some $\beta\geq 2$. Then by STEP 1 we know that \eqref{dieci} is true for all $a<a(\beta)$ (notice that $2<a(\beta)<q+1$). Choose $a>2$ and denote
$$\theta=\frac{q+1}{q-1}\frac{a-2}{a},$$
then $\theta\in (0,1)$ and using \eqref{seiBIS},  H\"older inequality, \eqref{dieci} and the interpolation inequality
, we obtain
\begin{eqnarray}
\|u_i(t)\|_{q+1}^{q+1}&&\leq C(1+\|u_i(t){u_i}_t(t)\|_1)\leq C(1+\|u_i(t)\|_{a}\|{u_i}_t(t)\|_{a'})\nonumber\\\label{daElevare}
&& \leq C(1+\|{u_i}_t(t)\|_{a'})
\leq C(1+\|{u_i}_t(t)\|_{\frac{q+1}{q}}^{\theta}\|{u_i}_t(t)\|_2^{1-\theta})\quad\mbox{ for all }\, t>0.
\end{eqnarray}

We now fix any $h\in (0,2)$ and  take $\tilde\beta=\beta+h.$ 
Than it is easy to show that for $a$ is sufficiently close to $a(\beta)$ one has 
\begin{equation}\label{sMaggUno} s:=\frac{2}{(1-\theta)\tilde\beta}>1.\end{equation} 
Notice also that $\theta\tilde\beta s'>1.$ 
We raise \eqref{daElevare} to the power $\tilde \beta$ and integrate it from $\delta$ to $T$, use H\"older inequality with exponents $s$ and $s'$, \eqref{cinque}, the maximal Sobolev regularity \eqref{MaximalSobolev} to get
\begin{eqnarray}\label{partenza}
&&\int_{\delta}^{T}\|u_i(t)\|_{q+1}^{\tilde\beta(q+1)}dt \leq C\left(1+\int_{\delta}^{T}\|{u_i}_t(t)\|_{\frac{q+1}{q}}^{\theta\tilde\beta}\|{u_i}_t(t)\|_2^{(1-\theta)\tilde\beta}dt\right)\nonumber\\
&&\nonumber\\
&& \qquad\qquad\quad \stackrel{\scriptsize{\mbox{(H\"older ineq.)}}}{\leq} C  \left[      1+
\left(\int_{\delta}^{T}\|{u_i}_t(t)\|_{\frac{q+1}{q}}^{\theta\tilde\beta s'}dt\right)^{\frac{1}{s'}}
\left(\int_{\delta}^{T}\|{u_i}_t(t)\|_{2}^{(1-\theta)\tilde\beta s}dt\right)^{\frac{1}{s}}  \right]  \nonumber\\
&& \qquad\qquad\qquad \;\;\; = C  \left[      1+
\left(\int_{\delta}^{T}\|{u_i}_t(t)\|_{\frac{q+1}{q}}^{\theta\tilde\beta s'}dt\right)^{\frac{1}{s'}}
\left(\int_{\delta}^{T}\|{u_i}_t(t)\|_{2}^{2}dt\right)^{\frac{1}{s}}  \right]   \nonumber\\
&&\qquad\qquad\qquad \;\; \stackrel{\scriptsize{\mbox{\eqref{cinque}}}}{\leq} C \left[ 1+ \left(\int_{\delta}^{T}\|{u_i}_t(t)\|_{\frac{q+1}{q}}^{\theta\tilde\beta s'}dt\right)^{\frac{1}{s'}}
\right]\nonumber\\ 
&&\quad\quad\:\qquad\qquad\;\; \stackrel{\scriptsize{\mbox{\eqref{MaximalSobolev}}}}{\leq} C \left[ 1+ \left(\int_{\delta}^{T}\|F_i(u_i(t))\|_{\frac{q+1}{q}}^{\theta\tilde\beta s'}dt +\|u_i(\delta)\|
_{2, \frac{q+1}{q}}^{\theta\tilde\beta s'}  \right)^{\frac{1}{s'}}\right] \nonumber\\
&&\qquad\qquad\qquad \quad\leq C \left[ 1+ \left(\int_{\delta}^{T}\|F_i(u_i(t))\|_{\frac{q+1}{q}}^{\theta\tilde\beta s'}dt  \right)^{\frac{1}{s'}}\right].
\end{eqnarray}
$$$$
Last substituting \eqref{usala}  into \eqref{partenza} we get

\begin{eqnarray*}
&&\int_{\delta}^{T}\|u_i(t)\|_{q+1}^{\tilde\beta(q+1)}dt \leq C\left[ 1+ \left(\int_{\delta}^{T}\|u_i(t)\|_{q+1}^{\frac{q\theta\tilde\beta s'}{q+1}(q+1)}dt\right)^{\frac{1}{s'}}\right],
\end{eqnarray*}

which implies the bound \eqref{ottoIMPR} with $\beta$ replaced by $\tilde\beta,$ since by our choice of $h$ and for $a$ sufficiently close to $a(\beta)$ one can show that 
\begin{equation}\label{condi}\theta s'\leq\frac{q+1}{q}.\end{equation}

%
%
%
%
%
%
%
\end{proof}
$$$$
$$$$
We are now in the position to conclude the proof of Proposition \ref{aPriori}.
$$$$
\begin{proof}[Proof of Proposition \ref{aPriori}. Case $T<\infty$]
$$$$

First we prove the bound \eqref{1StimaNormaMigliore}. From \eqref{sei} (indeed  when $i=3$ we are assuming $q\geq 3$) and using H\"older inequality we have 
$$\|u_i(t)\|_{1,2}^{2}\leq C(1+\|u_i(t){u_i}_t(t)\|_1)\leq C(1+\|u_i(t)\|_2\|{u_i}_t\|_2)\quad\forall t>0,$$
so, raising it to the power two and using \eqref{tre} it follows
$$\|u_i(t)\|_{1,2}^{4}\leq C(1+\|u_i(t)\|_2^2\|{u_i}_t\|_2^2)\leq C(1+\|{u_i}_t(t)\|_2^2)\quad\forall t>0.$$
Integrating, using \eqref{cinque} and the assumption $T<\infty$ we get 
\begin{equation}\label{caso3}\int_0^{T}\|u_i(t)\|_{1,2}^4dt\leq C\int_0^{T}(1+\|{u_i}_t(t)\|^2_2)dt<\infty.\end{equation}
Putting together \eqref{cinque} and \eqref{caso3} we have
\begin{equation}\label{f} \int_{0}^{T}\left(\|{u_i}_t(t)\|_2^{2}+\|u_i(t)\|_{1,2}^{4}\right)dt<\infty,
\end{equation}
and so from the interpolation embedding \eqref{interpolationTheorem2} we obtain the estimate \eqref{1StimaNormaMigliore}:
$$\sup_{(0,T)}\|u_i(t)\|_{a}<\infty \qquad\mbox{ for all } a<\frac{18}{5}.$$\\
\\
We now prove \eqref{2StimaNormaMigliore}. Hence we restrict to the cases $i=2,3$ and assume $q\geq 3$ when $i=2$ and $q>3$ when $i=3.$ 
The estimate \eqref{2StimaNormaMigliore} follows then from Lemma \ref{proofTinfinito}, taking $\beta$ big enough in \eqref{dieci}, since 
$a(\beta)\stackrel{\beta\rightarrow +\infty}{\longrightarrow} q+1.$
\end{proof}

$$$$
\section{Proof of theorems \ref{1TeoGlobalSolution}, \ref{2TeoGlobalSolution} and \ref{TeoGlobalSolution}}

The proof uses Amann's abstract results for global existence and relative compactness (see \cite{Amann}). Precisely we can apply to our case  \cite[Theorem 5.1]{Amann} combined together with \cite[Proposition 3.4]{Amann}.

In view of these general results, we have global solutions and boundedness in $W^{s,p}$-norm for every $s\in [1,2),$ provided that for every $\delta>0$
the following ``polynomial bound'' for $F_i(u_i(t))$ and  a priori bound for $\|u_i(t)\|_{p_0}$ hold
for certain  $1\leq \gamma_0<1+ \frac{2}{3}p_0:$

\begin{equation}\label{penu}
 \|F_i(u_i(t))\|_p\leq C\left(1+\||u_i(t)|^{\gamma_0}\|_p\right)\qquad t\in [\delta,T),
\end{equation}
\begin{equation}\label{ul}
 \sup_{[\delta,T)}\|u_i(t)\|_{p_0}<\infty.
\end{equation}
$$$$

When $i=1$ from  Lemma \ref{stimanonlinearitap}-\eqref{laPrima}  and Proposition \ref{aPriori}-\eqref{1StimaNormaMigliore}  we know that \eqref{penu} and \eqref{ul} are satisfied taking  $\gamma_0=3,$ and any $p_0\in (3,\frac{18}{5}).$ 
 \\ 
\\\\
Next we consider $i=2,3.$ In this case Lemma \ref{stimanonlinearitap}-\eqref{laSeconda}   gives \eqref{penu} with $\gamma_0=\max\{3,q\},$ hence it's enough to prove \eqref{ul} for any $p_0>\max\{3, \frac{3}{2}(q-1)\}.$
\\
\\
If $q<\frac{17}{5},$ namely if $\max\{3, \frac{3}{2}(q-1)\}<\frac{18}{5},$ then the conclusion follows directly from   Proposition \ref{aPriori}-\eqref{1StimaNormaMigliore}.

Precisely the estimate \eqref{1StimaNormaMigliore}, for any fixed $q\in (1, \frac{17}{5})$ in case $i=2$ or $q\in [3,\frac{17}{5})$ in case $i=3,$  implies \eqref{ul}  for any $p_0\in (\max\{3, \frac{3}{2}(q-1)\},\frac{18}{5}).$ \\
\\
In order to consider bigger values of $q$ we need to use Proposition \ref{aPriori}-  \eqref{2StimaNormaMigliore}.

Indeed for any fixed $q\in [\frac{17}{5},2^*-1)$  estimate \eqref{2StimaNormaMigliore} gives \eqref{ul} for  any $p_0\in(\frac{3}{2}(q-1),q+1),$ 
where $\frac{3}{2}(q-1)<q+1$
since $q<2^*-1$ by assumption.

$$$$

It remains to prove the relative compactness property. From the boundedness of the $W^{s,p}$-norm and 
the compactness of the embedding $W^{x,p}(\Omega)\hookrightarrow W^{y,p}(\Omega)$ for $y<x$ it follows that the set $\{u_i(t):t\geq\delta\}$ is relatively compact in $W^{s,p}(\Omega)$ for every $s\in [1,2).$ Hence the conclusion follows  from the embedding $W^{s,p}(\Omega)\hookrightarrow C^1(\bar\Omega)$ choosing $s$ sufficiently close to $2.$ 
$$$$
$$$$
\section*{Acknowledgments} 

The author would like to thank professor Tobias Weth for many enriching discussions and useful suggestions.

$$$$
$$$$

\end{document}